\documentclass{amsart}
\usepackage{amsfonts,amsmath,amssymb,amsthm}
\usepackage{enumitem}
\usepackage{tikz}
\usepackage{tikz-cd}
\usepackage{todonotes}
\usepackage{hyperref}

\newcommand{\bbz}{\mathbb{Z}}
\newcommand{\bbr}{\mathbb{R}}
\newcommand{\bbq}{\mathbb{Q}}
\newcommand{\bbf}{\mathbb{F}}
\newcommand{\bbp}{\mathbb{P}}
\newcommand{\ffp}{\mathfrak{p}}
\newcommand{\calO}{\mathcal{O}}
\newcommand{\calQ}{\mathcal{Q}}
\newcommand{\Cl}{\mathrm{Cl}}
\newcommand{\Gal}{\mathrm{Gal}}
\newcommand{\Zalgebra}[1]{\bbz_p[\![#1]\!]}
\newcommand{\coker}{\operatorname{coker}}
\newcommand{\Mat}{\mathrm{Mat}}
\newcommand{\rk}{\operatorname{rk}}
\newcommand{\End}{\mathrm{End}}

\newcommand{\len}{\ell}

\theoremstyle{plain}
\newtheorem{thm}{Theorem}[section]
\newtheorem{prop}[thm]{Proposition}
\newtheorem{lem}[thm]{Lemma}
\newtheorem{cor}[thm]{Corollary}
\theoremstyle{definition}
\newtheorem{defn}[thm]{Definition}
\newtheorem{exmp}[thm]{Example}
\newtheorem{rmk}[thm]{Remark}

\AtBeginEnvironment{exmp}{%
  \pushQED{\qed}
}
\AtEndEnvironment{exmp}{\popQED}

\usetikzlibrary{arrows.meta}

\title{Rank functions, \(\bbz_p\)-extensions and class numbers}
\author{Iq Souza}
\address{University of Manchester, Manchester, United Kingdom}
\email{henrique.souza@manchester.ac.uk}
\date{\today}

\begin{document}

\begin{abstract}
    If \(F\) is a number field and \(F = F_0 \subset F_1 \subset F_2\subset \cdots\) is a tower of Galois extensions with \(\Gal(F_n/F)\) cyclic of order \(p^n\), K. Iwasawa gave a closed formula for the growth of the \(p\)-part of the class number of \(F_n\) whenever \(n\) is large enough. We present a new proof of this formula, obtained by reinterpreting its summands in terms of Sylvester matrix rank functions on the Iwasawa algebra \(\Lambda\) without resorting to the classification of finitely-generated \(\Lambda\)-modules up to pseudo-isomorphism.
\end{abstract}

\maketitle

\section{Introduction}

The purpose of this expository note is to show how some of the results in the Iwasawa theory of \(\bbz_p\)-extensions of number fields can be interpreted in the language of Sylvester rank functions.

In the mid 1950's, Kenkichi Iwasawa was interested in the following question: given a tower \(F = F_0 \subset F_1 \subset F_2 \subset \cdots\) of number fields, how does the size of the class group \(\Cl(F_i)\) grow? Recall that the ideal class group \(\Cl(K)\) of a number field \(K\) with ring of integers \(\calO_K\) is defined as the quotient \(J_K/P_K\), where \(J_K\) is the group of fractional ideals of \(\calO_K\) and \(P_K\) is the subgroup of principal ideals. This is always a finite group, and it measures precisely the failure of unique factorization in \(\calO_K\).

In 1956, he gave a talk titled ``A theorem on Abelian groups and its application in algebraic number theory.'' on the Summer Meeting of the American Mathematical Society in Seattle. In his talk, he annouced the following remarkable theorem:

\begin{thm}[{\cite{iwasawa1959gamma}}]\label{thm-iwasawa-structure-theorem} Let \(F\) be a number field, and \(F = F_0 \subset F_1 \subset F_2 \subset \cdots\) be a tower of finite normal extensions of \(F\) such that for \(L = \bigcup_{i\geq 0} F_i\) we have \(\Gal(L/F) \simeq \bbz_p\). For each \(n\), let \(L_n\) be the subfield of \(L\) fixed by \(p^n\bbz_p\) and \(p^{e_n}\) be the highest power of \(p\) dividing the order of the class group \(\Cl(L_n)\). Then, there exist integers \(\nu\), \(\lambda\geq 0\) and \(\mu\geq 0\) such that \[e_n = \mu p^n + \lambda n + \nu\] for sufficiently large \(n\).
\end{thm}

The main algebraic ingredient in a standard modern proof of Iwasawa's theorem is a classification of finitely generated \(\Zalgebra{G}\)-modules. Taking \(G = \bbz_p\), although the completed group ring \(\Zalgebra{G}\) is not a principal ideal domain, it behaves very similarly to one. For every finitely generated \(\Zalgebra{G}\)-module \(M\), there are elements \(f_1\,,\ldots\,,f_n \in \Zalgebra{G}\) and a \(\Zalgebra{G}\)-module homomorphism \(\varphi\colon M \to \left(\bigoplus \Zalgebra{G}/(f_i)\right) \oplus \Zalgebra{G}^m\) such that \(\ker \varphi\) and \(\coker \varphi\) are both finite; this is the so called classification of \(\Zalgebra{G}\)-modules up to \emph{pseudo-isomorphism}. Thus, up to a finite constant error, it suffices to check Theorem~\ref{thm-iwasawa-structure-theorem} on modules of the form \(\Zalgebra{G}/(f_i)\). Such modules are very easy to describe since \(\Zalgebra{G}\) is a unique factorization domain.

Our goal is to give a proof of Theorem~\ref{thm-iwasawa-structure-theorem} without resorting to the classification of \(\Zalgebra{G}\)-modules up to pseudo-isomorphism. The way to go about this is to study three natural families of rank functions on ring \(\Zalgebra{G}\) and to realize the quantities appearing in Theorem~\ref{thm-iwasawa-structure-theorem} as the ``rescaled sides'' of a triangle formed by points in these three families (to made clear in Section~\ref{sec:iwasawa-triangle}). These families all accumulate in a common, canonical rank function on \(\Zalgebra{G}\), and computing the precise rate of convergence for this accumulation yields Iwasawa's formula.

The structure of this note is as follows: we collect in Section~\ref{sec-preliminaries} the definition of Sylvester rank functions and some basic examples, the facts about the Iwawasa algebra \(\Zalgebra{G}\) that will be used as well as the input from Global Class Field Theory that allows us to deduce information about the ideal class groups along the tower \(L/F\). In Section~\ref{sec:iwasawa-triangle}, we define the three families of Sylvester rank functions on \(\Zalgebra{G}\) which capture the quantities in Theorem~\ref{thm-iwasawa-structure-theorem}, and we prove this precise relation in Corollary~\ref{cor-Iwasawa-formula-arbitrary-module}. We then show how this Corollary immediately yields Theorem~\ref{thm-iwasawa-structure-theorem} for the case where \(F\) contains a single prime over \(p\), totally ramified in \(L\), in Section~\ref{sec-totally-ramified}. The adaptations requried to treat the general case are presented in Section~\ref{sec-general-extensions}.

\subsection*{Acknowledgements}

When I was writing my thesis during the 2025--2026 term, I got interested in Iwasawa's original paper as it is, perhaps, the first published result about the \(p\)-adic representation theory of a \(p\)-adic Lie group (the additive group \(\bbz_p\)). At the time, I got a feeling that his ideas could be reformulated in the language of Sylvester rank functions, but with the thesis deadline approaching, I could not afford the time to pursue this line of thought any further. Now, as a post-doc with plenty of time on my hands, I've decided to give this another shot and prepare these notes.

The spark to revisit this idea came from an invitation from A. Garrido to give a colloquium talk at the Universidad Complutense de Madrid. I knew immediately that would have to use this opportunity to spread the word of rank functions, but I was still lacking an interesting example to resonate with the one number-theorist in the audience. Thus, I'm indebted to both of them for inspiration to finally get on with it. I'd also like to thank J. Timmins for introducing me to T. Ochiai's book \cite{Ochiai2023Iwasawa1}. During the preparation of this note, the author was supported by the EPSRC grant UKRI2780.

\subsection*{Declaration of generative AI use}

During the preparation of this article, the author used ChatGPT (OpenAI, GPT-6) to discuss mathematical ideas and refine intermediate arguments, assist with literature searches, and proofread the manuscript. The original conception of the approach and the proof strategy are due to the author, who wrote and revised the manuscript. The author independently verified the mathematical arguments and references and takes full responsibility for the final content.

\section{Preliminaries}\label{sec-preliminaries}

\subsection{Sylvester rank functions}\label{sec-rank-functions}

For this subection, we let \(R\) denote an associative ring with unity, not necessarily commutative. We denote the space of all matrices of size \(n\times m\) over \(R\) by \(\Mat_{n\times m}(R)\). For square matrices of size \(n\times n\), we use the shorthand \(\Mat_n(R)\). The set \(\Mat(R)\) denotes the set of all matrices over \(R\) of any size, not necessarily square.

\begin{defn} A Sylvester matrix rank function is a map \(\rk\colon \Mat(R) \to \bbr_{\geq 0}\) assigning to each matrix \(A\) over \(R\) a non-negative real number \(\rk(A)\) satisfying:
\begin{enumerate}[label=(SMat\arabic*),wide]
    \item \(\rk(0) = 0\) and \(\rk(1) = 1\);
    \item \(\rk(AB) \leq \min\{\rk(A)\,,\rk(B)\}\) for any pair of matrices \(A\,,B \in \Mat(R)\) which can be multiplied;
    \item \(\rk\begin{pmatrix} A & 0\\ 0 & B\end{pmatrix} = \rk(A) + \rk(B)\) for any pair of matrices \(A\) and \(B\);
    \item \(\rk\begin{pmatrix}
        A & C\\ 0 & B
    \end{pmatrix} \geq \rk(A) + \rk(B)\) for any trio of matrices \(A\), \(B\) and \(C\) of the appropriate sizes.
\end{enumerate}
\end{defn}

If \(R\) is a field and \(\rk\) is the usual notion of rank from linear algebra (i.e. number of pivots in a row-echelon form), then axioms (SMat1) through (SMat4) hold. In fact, in this case, \(\rk\) is the \emph{unique} Sylvester matrix rank function defined on \(R\): it follows from (SMat2) that \(\rk(U) = n\) for any invertible matrix \(U \in \Mat_n(R)\), and again that \(\rk(UA) = \rk(AV) = \rk(A)\) for any matrix \(A\) and any pair \((U,V)\) of invertible matrices of the appropriate sizes. Thus, for an arbitrary rank function \(\rk\), \(\rk(A)\) is determined by the rank of any row-echelon form. This must in turn coincide with the number of pivots by (SMat1) and (SMat3).

Let \(\mathcal{M}\) denote the set of isomorphism classes of finitely presented right \(R\)-modules. Rank functions can also be defined in \(\mathcal{M}\):

\begin{defn} A Sylvester module rank function is a map \(\dim\colon \mathcal{M} \to \bbr_{\geq 0}\) assigning to each finitely presented right \(R\)-module \(M\) a non-negative real number \(\dim(M)\) satisfying:
\begin{enumerate}[label=(SMod\arabic*),wide]
    \item \(\dim(0) = 0\), \(\dim(R) = 1\);
    \item \(\dim(M \oplus N) = \dim(M) + \dim(N)\);
    \item For any exact sequence \[M' \to M \to M'' \to 0\] of finitely presented modules we have \[\dim(M'') \leq \dim M \leq \dim M' + \dim M''\,.\]
\end{enumerate}
\end{defn}

Observe that if \(R\) is Noetherian, then \(\mathcal{M}\) coincides with the class of all finitely generated right \(R\)-modules. This will be the case for all rings considered in this paper. There is a natural connection between matrix rank functions and module rank functions:

\begin{defn} Let \(\dim\) be a Sylvester module rank function, and for any matrix \(A \in \Mat_{m \times n}(R)\) define the associated matrix rank function \[\rk_{\dim}(A) = m - \dim(R^m/AR^n)\,.\] Conversely, if \(\rk\) is a Sylvester matrix rank function on \(R\) and \(M \simeq R^m/AR^n\) is a finitely-presented right \(R\)-module, define the associated module rank function \[\dim_{\rk}(M) = m - \rk(A)\,.\]
\end{defn}

\begin{prop}[{\cite[Prop. 3.1.4]{Souza2026Thesis}}] The associated module rank function \(\dim_{\rk}(M)\) does not depend on the choice of presentation \(M \simeq R^m/AR^n\), and is a Sylvester module rank function on \(R\). The associated matrix rank function \(\rk_{\dim}\) is also a Sylvester matrix rank function, and the maps \[\dim \mapsto \rk_{\dim}\,,\quad \rk \mapsto \dim_{\rk}\] induce bijections between the sets of Sylvester module rank functions and Sylvester matrix rank functions on any given ring \(R\).
\end{prop}

We shall denote by \(\bbp(R)\) the set of all Sylvester matrix/module rank functions on a ring \(R\), where an associated pair \(\rk/\dim\) represents the same element of \(\bbp(R)\). It is a convex, compact subspace of a real vector space. For any matrix \(A\) over \(R\), the function \[d_A(\rk,\rk') = |\rk(A) - \rk'(A)|\] defines a pseudo-metric on the set \(\bbp(R)\). We say that a sequence \(\rk_i \in \bbp(R)\) converges pointwise to a rank function \(\rk \in \bbp(R)\) if \(d_A(\rk_i,\rk) \to 0\) as \(i \to \infty\) for any matrix \(A\) over \(R\).

\begin{exmp}\label{exmp-rank-finite-len} Suppose that \(R\) is commutative and \(N \neq 0\) is an \(R\)-module of finite length \(\len_R(N)\). For any finitely presented \(R\)-module \(M\), we have that \(M \otimes_R N\) is again an \(R\)-module of finite length, and the map \[M \mapsto \dim_N(M) = \frac{\len_R(M\otimes_R N)}{\len_R(N)}\] is a Sylvester module rank function on \(R\). The associated Sylvester matrix rank function is defined by \[\rk_N(A) = \frac{\len_R(AN^n)}{\len_R N}\] for any matrix \(A \in \Mat_{m\times n}(R)\).
\end{exmp}

\begin{exmp}\label{exmp-pullback} Let \(\varphi\colon R \to S\) be a ring map and \(\rk\) be a Sylvester matrix rank function on \(S\). Then, \(\varphi^\#\rk(A) = \rk(\varphi(A))\) is a Sylvester matrix rank function on \(R\). If \(\dim\) is the associated Sylvester module rank function on \(S\), then \[\varphi^\#\dim(M) = \dim(M\otimes_R S)\] is the module rank function associated to \(\varphi^\#\rk\). In particular, if \(R\) is a subring of \(S\), the restriction of any rank function on \(S\) to \(\Mat(R)\) is again a rank function on \(R\).
\end{exmp}

\begin{exmp}\label{exmp-rank-fd-rep} Let \(\bbf\) be a field, \(V \neq 0\) be a finite-dimensional \(\bbf\)-vector space and \(\varphi\colon R \to \End_{\bbf}(V)\) a ring homomorphism. Then, for any matrix \(A\) over \(R\), the map \[\rk_V(A) = \frac{\rk_\bbf(\varphi(A))}{\dim_\bbf V}\] is a Sylvester matrix rank function on \(R\). It is the pullback of the unique Sylvester matrix rank function on \(\End_{\bbf}(V)\), the rank over \(\bbf\) normalized by \(\dim_\bbf V\).
\end{exmp}

\begin{exmp}[{\cite[Exmp. 3.1.10]{Souza2026Thesis}}]\label{exmp-rank-function-loc} Assume now that \(R\) is commutative domain and that \(S\) is a multiplicative set of non-zero elements in \(R\). If \(\rk\) is a Sylvester matrix rank function on \(R\), then it extends to a Sylvester matrix rank function on \(S^{-1}R\) if and only \(\rk(x) =1\) for any \(x \in S\). Moreover, this extension is unique.
\end{exmp}

A particular consequence of the phenomena in Example~\ref{exmp-rank-function-loc} is the following result:

\begin{prop}[{\cite[Cor. 5.2.5]{Souza2026Thesis}}]\label{prop-universal-rank-property} Let \(R\) be a commutative domain with field of fractions \(\calQ\), and \(\rk_i\) a sequence of Sylvester matrix rank functions on \(R\). Then, we have \(\lim_{i \to \infty} \rk_i = \rk_\calQ\) if and only if \[\lim_{i \to \infty} \rk_i(x) = 1\] for all non-zero \(x \in R\).
\end{prop}

\subsection{Iwasawa algebras}\label{sec-iwasawa-algebras}

Let \(\Lambda = \Zalgebra{\bbz_p}\) be the completed group ring of the pro-\(p\) group \(\bbz_p\). If we write \(\bbz_p = \overline{\langle \gamma \rangle}\) and \(T = \gamma - 1 \in \Lambda\), then we can identify \(\Lambda\) with the ring \(\Zalgebra{T}\) of formal power series in \(T\) over \(\bbz_p\). In particular, it is a local Noetherian domain with maximal ideal \((p,T)\). Recall that a distinguished polynomial \(f(T) \in \bbz_p[T]\) is a monic polynomial of the form \[f(T) = T^s + pa_{s-1}T^{s-1} + \cdots + pa_1T + pa_0\] for some coefficients \(a_0\,,\ldots\,,a_{s-1} \in \bbz_p\). In particular, the constant polynomial \(f(T) =1 \) is distinguished. The following result is known as an instance of the Weierstrass preparation theorem:

\begin{thm}[{\cite[Thm. 7.3]{Washington1997}}]\label{thm-Weierstrass-prep} For any non-zero \(f \in \Lambda\), there exists a unique distinguished polynomial \(h \in \bbz_p[T]\) and a unit \(u \in \Lambda^\times\) such that \[f = p^auh\text{ for some integer }a \geq 0\,.\]
\end{thm}

We will also make use of the Weierstrass division theorem for the power series ring \(\Lambda\).

\begin{prop}[{\cite[Prop. 7.2]{Washington1997}}]\label{prop-Weierstrass-div} For any \(f \in \Lambda\smallsetminus p\Lambda\), if \(f = uh\) with \(u \in \Lambda^\times\) and \(h\) distinguished, then \(\Lambda/(f)\) is a free \(\bbz_p\)-module of rank \(\deg h\) and basis \(1\,, T\,, \ldots\,, T^{\deg h -1}\).
\end{prop}

Since the quotient \(\Lambda/(p) \simeq \bbf_p[\![T]\!]\) is also a domain, the ideal \((p)\) is prime and therefore we can consider the localization \[\Lambda_{(p)} = \left\{\frac{f}{h}\colon f\,,h \in \Lambda\,, h \not\in (p)\right\}\,.\] We shall use the following facts about \(\Lambda_{(p)}\) and its quotients:

\begin{lem}\label{lem-localization-at-p-structure} The ring \(\Lambda_{(p)}\) is a discrete valuation ring with uniformizer \(p\) and residue field \(\bbf_p(\!(T)\!)\) of Laurent power series in \(T\) over \(\bbf_p\). If \(M\) is a finitely generated torsion \(\Lambda_{(p)}\)-module, then \(p^aM = 0\) for some integer \(a \geq 0\) and \[\len_{\Lambda_{(p)}}(M) = \sum_{i=0}^{a-1} \dim_{\bbf_p(\!(T)\!)} (p^iM/p^{i+1}M)\,.\] In particular, \(\len_{\Lambda_{(p)}}(\Lambda_{(p)}/(p^k)) = k\) for any \(k \geq 0\).
\end{lem}
\begin{proof} For an element \(\dfrac{f}{h}\) in \(\Lambda_{(p)}\) we can uniquely write \(f = p^af_0\) with \(f_0 \in \Lambda \smallsetminus p\Lambda\) and \(a \in \bbz_{\geq 0}\). In particular, \(f_0\) is a unit in \(\Lambda_{(p)}\), which proves that any element in \(\Lambda_{(p)}\) factors a power of \(p\) times a unit. If \(I\) is any non-zero ideal of \(\Lambda_{(p)}\), then choosing \(x =p^a u \in I\) with \(u \in \Lambda_{(p)}^\times\) and \(a\) minimal shows that \(I = (p^a)\): for any \(y = p^bu' \in I\), we have \(a \leq b\) and therefore \(y = p^a \cdot p^{b-a}u' \in (p^a)\). This shows that \(\Lambda_{(p)}\) is a discrete valuation ring with uniformizer \(p\). The residue field of \(\Lambda_{(p)}\) can then be identified with the fraction field of \(\Lambda/(p) \simeq \bbf_p[\![T]\!]\).

If \(M = \langle m_1\,,\ldots\,,m_k\rangle\) as a \(\Lambda_{(p)}\)-module, then since all the \(m_i\) are torsion elements, there exists a high enough power \(p^a\) of \(p\) that annihilates them: \(p^am_1 = \cdots = p^am_k = 0\). This proves that \(p^aM = 0\). This gives a finite filtration \[M \geq pM \geq p^2M \geq \cdots \geq p^{a-1}M \geq p^aM = 0\] whose successive quotients are finite-dimensional \(\bbf_p(\!(T)\!)\)-vector spaces. Since \(\bbf_p(\!(T)\!)\) is the only simple \(\Lambda_{(p)}\)-module, the claim about length follows.
\end{proof}

For every integer \(n \geq 0\), let \[\omega_n = \gamma^{p^n} - 1 = (1+T)^{p^n}-1 \in \Lambda\,.\] Observe that each \(\omega_n\) is a distinguished polynomial in \(T\) of degree \(p^n\). We can identify the quotient \(\Lambda/(\omega_n)\) with the group ring \(\bbz_p[\bbz_p/p^n\bbz_p] \simeq \bbz_p[C_{p^n}]\) of the cyclic group of order \(p^n\) over \(\bbz_p\).

\begin{lem}\label{lem-finite-cyclic-quotient-structure} For every pair \(n \geq 0\,, k \geq 1\) of integers, the ring \[\Lambda/(p^k,\omega_n) \simeq (\bbz/p^k\bbz)[C_{p^n}]\] is a local Artinian ring and a free \(\bbz/p^k\bbz\)-module of rank \(p^n\). Seen either as an \(\Lambda\)-module or as a \(\bbz_p\)-module, its length is \[\len_\Lambda(\Lambda/(p^k,\omega_n)) = \len_{\bbz_p}(\Lambda/(p^k,\omega_n)) = kp^n\,.\] For any finitely generated \(\Lambda/(p^k,\omega_n)\)-module \(M\), we have \[\len_{\Lambda/(p^k,\omega_n)}(M) = \len_\Lambda(M) = \len_{\bbz_p}(M) = \log_p|M| = \sum_{i=0}^{k-1} \dim_{\bbf_p}(p^iM/p^{i+1}M)\,.\]
\end{lem}
\begin{proof} The isomorphism \(\Lambda/(p^k,\omega_n) \simeq (\bbz/p^k\bbz)[C_{p^n}]\) is clear from the definitions. Since this is a finite ring, it is necessarily Artinian, and it is also local, being a quotient of the local ring \(\Lambda\). Its description as a group ring shows that it is free of rank \(p^n\) as a \(\bbz/p^k\bbz\)-module.

Since \(\Lambda \to \Lambda/(p^k,\omega_n)\) is surjective, the only simple \(\Lambda/(p^k,\omega_n)\)-module is \(\bbf_p\), which also the only simple \(\bbz_p\)-module, we get equality of lengths \[\len_{\Lambda/(p^k,\omega_n)}(M) = \len_\Lambda(M) = \len_{\bbz_p}(M)\] for any finitely generated \(\Lambda/(p^k,\omega_n)\)-module \(M\). The filtration \[M \geq pM \geq p^2M \geq \cdots \geq p^{k-1}M \geq p^kM = 0\] whose successive quotients are finite-dimensional \(\bbf_p\)-vector spaces proves that \[\len_{\bbz_p}(M) = \sum_{i=0}^{k-1} \dim_{\bbf_p} (p^iM/p^{i+1}M) = \log_p|M|\,.\] The formula for the length of the ring \(\Lambda/(p^k,\omega_n)\) now follows, since its order is \((p^k)^{p^n}\).
\end{proof}

We will also use the following result, which is a particular case of a theorem of M. Harris (\cite{Harris2000}):

\begin{prop}[{cf. \cite[Thm. 5.2.6]{Souza2026Thesis}}]\label{prop-Harris-theorem} Let \(M\) be a finitely generated \(\Lambda\)-module, and for each \(n\geq 0\), put \[\dim_n M = \frac{\dim_{\bbq_p} ((M/\omega_nM) \otimes_{\bbz_p} \bbq_p)}{p^n}\,.\] Then, the \(\dim_n\) are Sylvester module rank functions on \(\Lambda\), and \(\lim_{n \to \infty} \dim_n = \dim_{\Lambda}\) where \(\dim_\Lambda(M) = \dim_{\calQ} (M \otimes_\Lambda \calQ)\) and \(\calQ\) is the field of fractions of \(\Lambda\).
\end{prop}
\begin{proof} For each \(n\geq 0\), consider the finite-dimensional \(\bbq_p\)-vector space \(V_n = \bbq_p[G/G^{p^n}]\) of dimension \(p^n\). The left action of \(G\) on the quotient group \(G/G^{p^n}\) induces a ring homomorphism \(\Lambda \to \End_{\bbq_p}(V_n)\), and by definition the functions \(\dim_n\) are the pullbacks of the normalized dimension in \(\End_{\bbq_p}(V_n)\) as in Example~\ref{exmp-rank-fd-rep}. In particular, they are Sylvester module rank functions, and we denote by \(\rk_n\) the associated Sylvester matrix rank function.

To prove that \(\lim_{n \to \infty} \dim_n = \dim_\Lambda\), it suffices to prove that \(\rk_n(x) \to 1\) as \(n \to \infty\) for any non-zero \(x \in \Lambda\) by Proposition~\ref{prop-universal-rank-property}. Using Theorem~\ref{thm-Weierstrass-prep}, we can write \(x = p^auh\) where \(a \geq 0\), \(u \in \Lambda^\times\) is a unit and \(h \in \bbz_p[T]\) is a distinguished polynomial. In particular, \(\rk_n(x) = \rk_n(h)\). Note that by Proposition~\ref{prop-Weierstrass-div}, \(\bbq_p[G/G^{p^n}] \simeq \Lambda/(\omega_n) \otimes_{\bbz_p} \bbq_p\) has \(\{1\,,T\,,\ldots\,,T^{p^n-1}\}\) as a basis. If \(k = \deg h\), then multiplication by \(h\) is injective on the subspace spanned by \(\{1\,,T\,,\ldots\,,T^{p^n-1-k}\}\), which gives us the bounds \[1 \geq \rk_n(h) \geq \frac{p^n-k}{p^n}\,.\] As \(n\) goes to infinity, the right hand side converges to \(1\), proving the claim. Hence, as in Example~\ref{exmp-rank-function-loc}, we have \(\lim_{n \to \infty} \dim_n = \dim_\Lambda\) as desired.
\end{proof}

\subsection{Global Class Field Theory}\label{sec-GCFT}

Let \(K\) be a number field, that is, a finite extension of \(\bbq\), and denote by \(\calO_K\) its ring of integers. A fractional ideal of \(K\) is a subset \(I \subset K\) of the form \(r^{-1}J\), where \(J\) is an ideal of \(\calO_K\). It is invertible if there exists another fractional ideal \(J \subset K\) such that \(IJ =\calO_K\). The set of all invertible fractional ideals forms an abelian group \(J_K\) under multiplication, and it contains the subgroup \(P_K\) of all principal fractional ideals \(\calO_K x\) for non-zero \(x \in K\). The quotient \(\Cl(K) = J_K/P_K\) forms an abelian group called the \emph{ideal class group} of \(K\). The ring \(\calO_K\) is a principal ideal domain if and only if \(\Cl(K) = 1\). The class number of \(K\) is the order \(h(K) = |\Cl(K)|\) of its ideal class group, and is a measure of how much does unique factorization fails in the ring \(\calO_K\).

Assume that \(L/K\) is a Galois extension of number fields with Galois group \(G = \Gal(L/K)\). Recall the existence of a multiplicative Galois norm map \(N_{L/K}\colon L \to K\) given by \[N_{L/K}(x) = \prod_{\sigma \in G} \sigma(x)\,.\] This norm map induces a homomorphism \(N_{L/K}\colon J_L \to J_K\) between the groups of fractional ideals by declaring \(N_{L/K}(I)\) to be the fractional ideal of \(K\) generated by the set \(\{N_{L/K}(x)\colon x \in I\}\) for every fractional ideal \(I\) of \(L\). This map sends principal fractional ideals to principal fractional ideals, and thus descends to a homomorphism \(N_{L/K}\colon \Cl(L) \to \Cl(K)\) of ideal class groups.

Now, let \(F = F_0 \subset F_1 \subset F_2 \subset \cdots \) be a tower of finite Galois extensions of a number field \(F\) such that for \(L = \bigcup_{i\geq 0} F_i\) we have \(G = \Gal(L/F) \simeq \bbz_p\). For each \(n\), let \(L_n\) be the subfield of \(L\) fixed by \(p^n\bbz_p\). For an abelian group \(A\), denote by \(A(p)\) its \(p\)-primary part, that is, the set of all elements annihilated by some power of \(p\). With the norm maps defined above, the class groups \(\Cl(L_n)\) fit into an inverse system of abelian groups, and we let \[X = \varprojlim \Cl(L_n)(p)\,.\] Since each \(L_n\) is a Galois extension of \(F\), the Galois group \(G\) acts continuously and compatibly with the norm maps on each \(\Cl(L_n)\) endowed with the discrete topology.

We summarize in the following theorem the main results about \(X\) that will be used to prove Theorem~\ref{thm-iwasawa-structure-theorem}.

\begin{thm}\label{thm-GCFT}Let \(\gamma\) be a topological generator of \(G \simeq \bbz_p\), and set \(\omega_n = \gamma^{p^n} - 1 \in \Zalgebra{G}\). For any pair of non-negative integers \(n\geq m\), let \(\eta_{n,m} = \dfrac{\omega_n}{\omega_m}\). 
\begin{enumerate}[label=(\arabic*)]
    \item {\upshape (\cite[Thm. I.6.3]{Neukirch1999}, \cite[Lem. 2.23]{Ochiai2023Iwasawa1})} Each \(\Cl(L_n)\) is a finite group, so \(X\) is a pro-\(p\) abelian group with a continuous \(G\) action. In particular, \(X\) is a \(\Zalgebra{G}\)-module.
    \item {\upshape (\cite[Prop. VI.6.9]{Neukirch1999})} The Artin symbol map identifies \(\Cl(L_n)(p)\) with the Galois group of the maximal unramified abelian pro-\(p\) extension of \(L_n\) in which all real places split completely, and \(X\) with the Galois group \(\Gal(H/L)\) of the maximal unramified abelian pro-\(p\) extension \(H\) of \(L\) in which all real places split completely.
    \item {\upshape (\cite[Prop. 13.2 and Lem. 13.3]{Washington1997})} In each extension \(L/L_n\), only primes above \(p\) ramify, and at least one prime ramifies. Moreover, for large enough \(n\) all ramified primes in \(L/L_n\) are totally ramified.
    \item For each prime \(v\) in \(L_n\) above \(p\), let \(I_v \leq G^{p^n}\) be the inertia group of \(v\) under the extension \(L/L_n\). If \(\Sigma_n\colon \bigoplus_{v \mid p} I_v \to G^{p^n}\) denotes the map induced by the inclusions \(I_v \leq G^{p^n}\), then there is an exact sequence 
    \begin{equation}
        \ker \Sigma_n \to X/\omega_n X \to \Cl(L_n)(p) \to \coker \Sigma_n \to 0\,.
    \end{equation}
    \item There exists a positive integer \(m\) such that for any \(n \geq m\) the maps \(X \to \Cl(L_n)(p)\) are surjective. Moreover, if \(Y_n = \ker(X \to \Cl(L_n)(p))\), we have \(Y_n = \eta_{n,m}Y_m\).
\end{enumerate}
\end{thm}
\begin{proof} We will only prove (4) and (5), since the first three are explicitly stated in the cited references. For the proof of (4), we essentially follow the ideas in \cite[pp. 46--47]{Ochiai2023Iwasawa1}. Let \(H\) be the maximal unramified abelian pro-\(p\) extension of \(L\) in which all real places split completely, for each \(n \geq 0\) put \(E_n = \Gal(H/L_n)\) and \(G_n = \Gal(L/L_n) = G^{p^n}\). The short exact sequence \[1 \to X \simeq \Gal(H/L) \to \Gal(H/F) \to G \to 1\] coming from (2) also gives us the short exact sequence \[1 \to X \to E_n \to G_n \to 1\] of profinite groups for each \(n \geq 0\). This sequence splits, and upon passing to the abelianization, we have the short exact sequence \[0 \to X/\omega_nX \to E_n^{\text{ab}} \to G_n \to 0\,.\]

For a prime \(v\) in \(L_n\) above \(p\), choose a prime \(w\) in \(H\) above \(v\) and let \(u\) be its restriction to \(L\). Let \(\widetilde{I_v}\) be the inertia group \(I(w/v) \leq E_n\) and \(I_v\) the inertia group \(I(u/v) \leq G_n\). Observe that \(\widetilde{I_v}\) only depends on \(w\) up to conjugation in \(E_n\), and \(I_v\) is independent of \(w\) by the fact that \(G_n\) is abelian. Restriction gives us a surjective \(\widetilde{I_v} \to I_v\) whose kernel is the inertia group \(I(w/u) \leq X\) by \cite[Prop. II.9.5]{Neukirch1999}. However, \(H/L\) is totally unramified, so the inertia group \(I(w/u)\) is trivial. In particular, we can lift \(\Sigma_n\) to a map \(\iota_n\colon \bigoplus_{v \mid p} I_v \to E_n^{\text{ab}}\) along the isomorphisms \(\widetilde{I_v} \to I_v\).

By the definition of the inertia subgroups, \(\operatorname{im}(\iota_n)\) is the subgroup of \(E_n^{\text{ab}}\) corresponding to maximal abelian subextension of \(H/L_n\) unramified at all primes above \(p\). By (3), this means that this extension is completely unramified, which by (2) yields a natural isomorphism \[E_n^{\text{ab}}/\operatorname{im}(\iota_n) \simeq \Cl(L_n)(p)\,.\] Applying the snake lemma to the commutative diagram 
\[\begin{tikzcd}
0 \arrow[r] & 0 \arrow[r] \arrow[d]  & \bigoplus_{v \mid p} I_v \arrow[r, "\operatorname{id}"] \arrow[d, "\iota_n"] & \bigoplus_{v \mid p} I_v \arrow[r] \arrow[d, "\Sigma_n"] & 0 \\
0 \arrow[r] & X/\omega_n X \arrow[r] & E_n^{\text{ab}} \arrow[r]                                                    & G_n \arrow[r]                                            & 0
\end{tikzcd}\] gives the long exact sequence \[0 \to \ker \iota_n \to \ker \Sigma_n \to X/\omega_nX \to \Cl(L_n)(p) \to \coker(\Sigma_n) \to 0\,,\] from which (4) follows.

For (5), let \(m\) be such that all ramified primes in \(L/L_m\) are totally ramified, which exists by (3). Now apply \cite[Lem. 13.15]{Washington1997} to the extension \(L/L_m\).
\end{proof}

\section{The Iwasawa rank triangle}\label{sec:iwasawa-triangle}

In this section, we keep the notations of Sections~\ref{sec-rank-functions} and~\ref{sec-iwasawa-algebras}. In particular, \(G = \bbz_p = \overline{\langle \gamma \rangle}\) denotes the infinite pro-cyclic pro-\(p\) group and \(\Lambda = \Zalgebra{G} \simeq \Zalgebra{T}\) is its completed group algebra, where \(T = \gamma - 1\). We also denote by \(\omega_n\) the element \(\gamma^{p^n}-1 \in \Lambda\). We are interested in three families of Sylvester matrix rank functions on the Iwasawa \(\Lambda\), which will form our ``Iwasawa rank triangle''.

\begin{enumerate}[label=(\Alph*)]
    \item The rank \(\rk_\Lambda\) coming from the embedding of \(\Lambda\) into its field of fractions \(\calQ = \operatorname{Frac}(\Lambda)\). On a \(\Lambda\)-module \(M\), it is given by \[\dim_{\Lambda} M = \dim_{\calQ} M \otimes_{\Lambda} \calQ\,.\]
    \item For any integer \(k \geq 1\), the \(\Lambda_{(p)}\)-module \(\Lambda_{(p)}/(p^k)\) has length \(k\) by Lemma~\ref{lem-localization-at-p-structure}, and thus induces a rank function \(\rk_{(p),k}\) on \(\Lambda_{(p)}\) as in Example~\ref{exmp-rank-finite-len}. We also denote by \(\rk_{(p),k}\) its restriction to \(\Lambda\). On a \(\Lambda\)-module \(M\), it is given by \[\dim_{(p),k}(M) = \frac{\len_{\Lambda_{(p)}}(M_{(p)}/p^kM_{(p)})}{k}\,.\]
    \item For any pair of integers \(k \geq 1\) and \(n \geq 0\), the ring \(\Lambda/(\omega_n,p^k)\) is local Artinian of length \(kp^n\) by Lemma~\ref{lem-finite-cyclic-quotient-structure}, and thus we have a rank function \(\rk_{n,k}\) on \(\Lambda/(\omega_n,p^k)\) as in Example~\ref{exmp-rank-finite-len}. We also denote by \(\rk_{n,k}\) its pullback to \(\Lambda\) under the projection \(\Lambda \to \Lambda/(\omega_n,p^k)\), and on a \(\Lambda\)-module \(M\) it is given by \[\dim_{n,k}(M) = \frac{\len_{\Lambda/(\omega_n,p^k)}(M/(\omega_nM + p^kM))}{kp^n} = \frac{\log_p|M/(\omega_nM + p^kM)|}{kp^n}\,.\]
\end{enumerate}

By Proposition~\ref{prop-universal-rank-property}, we have that a sequence \(\rk_i\) of rank functions in \(\bbp(\Lambda)\) converges to \(\rk_\Lambda\) if, and only if, \(\lim_{i\to \infty} \rk_i(x) = 1\) for any non-zero \(x \in \Lambda\). It is not hard to see that this is the case for the families \(\rk_{(p),k}\) and \(\rk_{n,k}\) as \(n\,,k \to \infty\).

\begin{lem}\label{lem-topology-iwasawa-triangle} As rank functions on \(\Lambda\), we have:
\[\lim_{k \to \infty} \rk_{(p),k} = \rk_\Lambda\,,\]
\[\lim_{n \to \infty} \rk_{n,k} = \rk_{(p),k}\,,\]
\[\lim_{k\to \infty} \left(\lim_{n \to \infty} \rk_{n,k}\right) = \rk_\Lambda\,.\]
\end{lem}
\begin{proof} Note that the third limit follows from the first two. For the first limit, it suffices to show that for each non-zero \(x \in \Lambda\) we have \(\rk_{(p),k}(x) \to 1\) as \(k \to \infty\). By Theorem~\ref{thm-Weierstrass-prep}, we can write \(x = p^auh\) where \(a \geq 0\) is an integer, \(u \in \Lambda^\times\) is a unit and \(h \in \bbz_p[T]\) is a distinguished polynomial. Observe that both \(u\) and \(h\) are units in the ring \(\Lambda_{(p)}\), and therefore for \(k\geq a\):
\begin{align*}
    \rk_{(p),k}(x) & = \rk_{(p),k}(p^a)\\
    &= \frac{\len_{\Lambda_{(p)}}(p^a\Lambda_{(p)}/p^k\Lambda_{(p)})}{k}\\
    &= \frac{k-a}{k}\,. 
\end{align*}
This clearly converges to \(1\) as \(k \to \infty\).

For the second limit, by the compactness of \(\bbp(\Lambda)\) choose a convergent subnet \(\{\rk_i = \rk_{n_i,k}\}_{i \in I}\) with \(k\) fixed, and let \(\rk = \lim_{i \in I} \rk_i\). Denote by \(\dim_i\) and \(\dim\) the associated Sylvester module rank functions. As in Example~\ref{exmp-rank-function-loc}, if we show that \(\dim(\Lambda/(x)) = 0\) for any \(x \in \Lambda \smallsetminus (p)\), then \(\dim\) extends to a rank function on \(\Lambda_{(p)}\). By Theorem~\ref{thm-Weierstrass-prep}, any such \(x\) can be written as \(x = uh\) where \(u \in \Lambda^\times\) is a unit and \(h \in \bbz_p[T]\) is a distinguished polynomial, and moreover \[\dim_i(\Lambda/(x)) = \dim_i(\Lambda/(h)) = \frac{\len_{\Lambda/(\omega_{n_i},p^k)}(\Lambda/(\omega_{n_i},p^k,h))}{kp^{n_i}} = \frac{\log_p |\Lambda/(\omega_{n_i},p^k,h)|}{kp^{n_i}}\,.\] By Proposition~\ref{prop-Weierstrass-div}, \(\Lambda/(h,p^k)\) is a free \(\bbz/p^k\bbz\)-module of rank \(\deg h\). Thus:
\begin{align*}
    \dim_i(\Lambda/(x)) &= \frac{\log_p|\Lambda/(\omega_{n_i},p^k,h)|}{kp^{n_i}}\\
    &\leq \frac{k\deg h}{kp^{n_i}} = \frac{\deg h}{p^{n_i}}\to 0\,.
\end{align*}

Hence, \(\dim\) (and, equivalently, \(\rk\)) extends to \(\Lambda_{(p)}\), which is a principal ideal domain. Since any matrix \(A\) over \(\Lambda_{(p)}\) can be put into a Smith normal form \[UAV = \begin{pmatrix}
    p^{a_1} & 0 & 0 & \cdots & 0 & 0\\
    0 & p^{a_2} & 0 & \cdots & 0 & 0\\
    0 & 0 & p^{a_3} &  \cdots & 0 & 0\\
    \vdots & \vdots & \vdots & \ddots & \vdots & \vdots \\
    0 & 0 & 0 & \cdots & 0 & 0\\
    0 & 0 & 0 & \cdots & 0 & 0
\end{pmatrix}\quad\text{ with }U\,, V\text{ invertible over }\Lambda_{(p)}\,,\] it suffices to prove that \(\lim_{i \in I} \rk_i(p^a) = \rk_{(p),k}(p^a) = \dfrac{\max\{0,k-a\}}{k}\). This is immediate, since
\begin{align*}
    \rk_i(p^a) &= 1 - \frac{\log_p|\Lambda/(\omega_{n_i},p^k,p^a)|}{kp^{n_i}}\\
    &= 1 - \frac{\log_p|\Lambda/(\omega_{n_i},p^{\min\{k,a\}})|}{kp^{n_i}}\\
    &= 1- \frac{\min\{k,a\}p^{n_i}}{kp^{n_i}}\\
    &= \frac{\max\{0,k-a\}}{k}
\end{align*}
is constant as a function of \(i \in I\). Since the convergent subnet \(\{\rk_i\}_{i \in I}\) was arbitrary, we have that \(\rk_{n,k} \to \rk_{(p),k}\) as \(n \to \infty\).
\end{proof}

Using Lemma~\ref{lem-topology-iwasawa-triangle}, we represent in Figure~\ref{fig:Iwasawa-triangle} how our three families of Sylvester rank functions sit inside the space \(\bbp(\Lambda)\). The shape of this figure is our so called \emph{Iwasawa rank triangle}. Our strategy for proving Theorem~\ref{thm-iwasawa-structure-theorem} will be to use the pseudo-metrics \[d_A(\rk,\rk') = |\rk(A) - \rk'(A)|\] of Section~\ref{sec-rank-functions} on the space \(\bbp(\Lambda)\) to measure all three ``sides'' of this triangle, where \(A\) will be an appropriately chosen matrix over \(\Lambda\) related to the class groups \(\Cl(F_i)\).

\begin{figure}[!ht]
    \centering
    \begin{tikzpicture}[
  x=1.25cm,y=1.25cm,
  dot/.style={circle,fill=black,inner sep=1.35pt},
  endpoint/.style={circle,fill=black,inner sep=1.9pt},
  guide/.style={densely dotted,black!35,line width=0.45pt},
  direction/.style={-{Stealth[length=2mm]},line width=0.55pt}
]
  \coordinate (G) at (0,5);
  \coordinate (P) at (0,0);
  \coordinate (R) at (5,0);

  \draw[guide] (P) -- (G);
  \foreach \y in {1.15,2.10,2.85,3.45,3.92,4.28,4.55,4.74}
    \node[dot] at (0,\y) {};

  \draw[guide] (R) .. controls (2.8,1.1) and (0.75,2.5) .. (G);
  \foreach \t in {0.14,0.28,0.42,0.55,0.67,0.77,0.85,0.91,0.95}
    \path (R) .. controls (2.8,1.1) and (0.75,2.5) .. (G)
      node[pos=\t,dot] {};

  \draw[guide] (P) -- (R);
  \foreach \x in {3.65,2.55,1.70,1.08,0.65,0.36,0.17}
    \node[dot] at (\x,0) {};

  \node[endpoint,label=above:{$\rk_{\Lambda}$}] at (G) {};
  \node[endpoint,label=below:{$\rk_{(p),k}$}] at (P) {};
  \node[endpoint,label=below:{$\rk_{n,k}$}] at (R) {};

  \draw[direction] (-0.8,0.5) -- (-0.8,4.5)
    node[midway,left=3pt] {$k\to\infty$};
  \draw[direction] (4.5,-0.85) -- (0.5,-0.85)
    node[midway,below=3pt] {$n\to\infty\quad(k\text{ fixed})$};
  \node at (2.45,3.5) {$n,k\to\infty$};
\end{tikzpicture}
    \caption{The distribution of the rank functions \(\rk_\Lambda\), \(\rk_{(p),k}\) and \(\rk_{n,k}\) in the space \(\bbp(\Lambda)\) which form the Iwasawa rank triangle.}
    \label{fig:Iwasawa-triangle}
\end{figure}
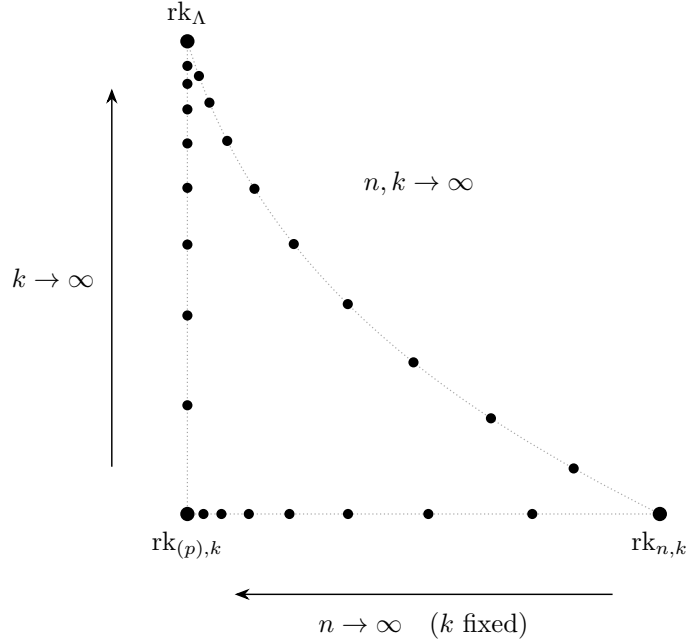

A remarkable fact about the Iwasawa rank triangle in Figure~\ref{fig:Iwasawa-triangle} is that, when measuring its sides under the pseudo-metrics \(d_A\), it always degenerates.

\begin{thm}\label{thm-main-triangle-inequality} Let \(A\) be any matrix over \(\Lambda\) and \(d_A\) be the corresponding pseudo-metric on the space \(\bbp(\Lambda)\).  Then, for any pair \(n\geq 0\) and \(k\geq 1\) of integers we have \[d_A(\rk_\Lambda,\rk_{n,k}) = d_A(\rk_\Lambda, \rk_{(p),k}) + d_A(\rk_{(p),k},\rk_{n,k})\,.\]
\end{thm}
\begin{proof} It suffices to prove the inequalities \[\rk_{n,k}(A) \leq \rk_{(p),k}(A) \leq \rk_\Lambda(A)\] for any matrix \(A \in \Mat_{r\times s}(\Lambda)\) and any pair \(n\geq 0\,, k\geq 1\) of integers. We will first prove that \(\rk_\Lambda\) is the \emph{universal} rank function on \(\bbp(\Lambda)\) (\cite[Def. 3.2.1]{Souza2026Thesis}), that is, \(\rk_\Lambda\) is the greatest element of \(\bbp(\Lambda)\) under the pointwise partial order: \(\rk_\Lambda(A) \geq \rk'(A)\) for any matrix \(A\) and any rank function \(\rk' \in \bbp(\Lambda)\) (see also \cite[Exmp. 4.1.5]{Souza2026Thesis}).

Let \(m_A\colon \Lambda^s \to \Lambda^r\) be the map given by left multiplication by \(A\), and also denote by \(m_A^\calQ\colon \calQ^s \to \calQ^r\) its extension to \(\calQ\)-vector spaces. We see \(\Lambda^s\) and \(\Lambda^r\) as naturally embedded inside \(\calQ^s\) and \(\calQ^r\) respectively, such that the restriction of \(m_A^\calQ\) to \(\Lambda^s\) is \(m_A\). Let \(L = m_A(\calQ^s) \cap \Lambda^r\) and \(t = \rk_\Lambda(A)\), so that \(\dim_{\calQ} m_A(\calQ^s) = t\). Since \(\Lambda\) is Noetherian, \(L\) is a finitely generated \(\Lambda\)-submodule of \(\Lambda^r\), and \(L \cap p\Lambda^r = pL\).

Hence, \(L/pL\) is a \(\Lambda/(p)\)-submodule of \((\Lambda/(p))^r\). As \(\Lambda/(p) \simeq \bbf_p[\![T]\!]\) is a discrete valuation ring, \(L/pL\) must be free as a \(\Lambda/(p)\)-module, say of rank \(l\). If we choose lifts \(x_1\,,\ldots\,,x_l\) in \(L\) of a basis of \(L/pL\), by Nakayama's lemma they generate the \(\Lambda\)-module \(L\). Then, the lifts \(x_1\,,\ldots\,,x_l\) are also linearly independent over \(\calQ\): clearing out the denominators and the common powers of \(p\) in any relation \[a_1x_1 + \cdots + a_lx_l = 0\] over \(\calQ\) descends to a linear relation between their images in \(L/pL\) with non-zero coefficients, a contradiction. Hence, \(L\) is a free \(\Lambda\)-module, necessarily of rank \(t = l\). Since \(m_A(\Lambda^s) \leq L \simeq \Lambda^t \leq \Lambda^r\), we get a factorization \(A = BC\) where \(B \in \Mat_{r \times t}(\Lambda)\) and \(C \in \Mat_{t \times s}(\Lambda)\). If \(\rk\) is any Sylvester rank function on \(\Lambda\), by (SMat2) we have \[\rk(A) = \rk(BC) \leq \min\{\rk(B),\rk(C)\} \leq t = \rk_\Lambda(A)\,,\] as desired.

For the other inequality, we consider the subring \(\Lambda_n = \Zalgebra{G^{p^n}} = \Zalgebra{\omega_n}\) of \(\Lambda\). Then, using the coset decomposition \(G = \bigcup_{g \in G/G^{p^n}} gG^{p^n}\), we see that \(\Lambda\) is a free \(\Lambda_n\)-module of rank \(p^n\) under the natural embedding \(\Lambda_n \to \Lambda\), so we can see \(A\) as a matrix of size \(p^nr\times p^ns\) over \(\Lambda_n\). If we take \(\widetilde{M} = \Lambda^r/A\Lambda^s\) and \(M = \widetilde{M}/p^k\widetilde{M} = \Lambda^r/B\Lambda^{s'}\) for some matrix \(B \in \Mat_{r\times s'}(\Lambda)\) we have \[\rk_{n,k}(A) = r - \frac{\len_{\Lambda} (M/\omega_nM)}{kp^n}\,,\quad \rk_{(p),k}(A) = r - \frac{\len_{\Lambda_{(p)}}(M_{(p)})}{k}\,.\] Thus, the inequality \(\rk_{n,k}(A) \leq \rk_{(p),k}(A)\) is equivalent to \begin{equation}\label{eq-len-inequality-1}\len_{\Lambda}(M/\omega_nM) \geq p^n\len_{\Lambda_{(p)}}(M_{(p)})\,.\end{equation}

On one hand, \[\len_{\Lambda}(M/\omega_nM) = \log_p|M/\omega_nM| = \len_{\Lambda_n}(M/\omega_nM) = \len_{\bbz_p}(M/\omega_nM)\,.\] On the other hand, \(\Lambda_{(p)}\) is a free \((\Lambda_n)_{(p)}\)-module of rank \(p^n\) and the corresponding residue field extension \(\bbf_p(\!(T)\!)/\bbf_p(\!(\omega_n)\!)\) has degree \(p^n\), so that \begin{align*}
\len_{(\Lambda_n)_{(p)}}(M_{(p)}) &= \sum_{i=0}^{k-1} \dim_{\bbf_p(\!(\omega_n)\!)} (p^iM_{(p)}/p^{i+1}M_{(p)})\\
&= \sum_{i=0}^{k-1} p^n\dim_{\bbf_p(\!(T)\!)}(p^iM_{(p)}/p^{i+1}M_{(p)})\\
&= p^n\len_{\Lambda_{(p)}}(M_{(p)})\,.
\end{align*} Thus, by the isomorphism \((\Lambda_n)_{(p)} \otimes_{\Lambda_n} \Lambda \simeq \Lambda_{(p)}\), the inequality~(\ref{eq-len-inequality-1}) is equivalent to \begin{equation}\label{eq-len-inequality-2}
    \len_{\Lambda_n}(M/\omega_nM) \geq \len_{(\Lambda_n)_{(p)}}(M_{(p)})\,.
\end{equation}

For each square submatrix \(C\) of \(B\) (again, seen as a matrix over \(\Lambda_n\) of size \(p^nr \times p^ns'\)), its determinant \((\det C)(\omega_n)\) is a power series over \(\bbz_p\) in the variable \(\omega_n\). Since both \(\Lambda_n/(\omega_n) \simeq \bbz_p\) and \((\Lambda_n)_{(p)}\) are discrete valuation rings with respective valuations \(v_p\) and \(v_{(p)}\), by taking Smith normal forms we have the identities: \[\len_{\bbz_p}(M/\omega_nM) = \min\{v_p((\det C)(0)) \colon C\text{ is a maximal square submatrix of }B\}\,,\]
\[\len_{(\Lambda_n)_{(p)}}(M_{(p)}) = \min\{v_{(p)}((\det C)(\omega_n))\colon C\text{ is a maximal square submatrix of }B\}\,.\] The inequality~(\ref{eq-len-inequality-2}) now follows, since \(v_{(p)}((\det C)(\omega_n)) = l\) implies that \(p^l\) divides all the coefficients of \((\det C)(\omega_n)\), so in particular \(p^l\) divides its constant coefficient \((\det C)(0)\).
\end{proof}

For finitely generated torsion \(\Lambda\)-modules \(M\), Theorem~\ref{thm-main-triangle-inequality} already gives some information on the growth of coinvariant quotients:

\begin{cor}\label{cor-Iwasawa-for-torsion} Let \(M \simeq \Lambda^r/A\Lambda^s\) be a finitely generated torsion \(\Lambda\)-module. Then, for any integer \(n \geq 0\) and \(\mu = \len_{\Lambda_{(p)}}(M_{(p)})\) we have \[\log_p|M/(\omega_nM + p^kM)| = \mu p^n + kp^nd_A(\rk_{(p),k},\rk_{n,k})\] whenever \(k \geq 1\) is large enough.
\end{cor}
\begin{proof} Note that if \(\rk\,,\rk'\) are Sylvester matrix rank functions on \(\bbp(\Lambda)\) with associated Sylvester module rank functions \(\dim\,,\dim'\), then \[d_A(\rk,\rk') =|\rk(A) - \rk'(A)| = |\dim(M) - \dim'(M)|\,.\] Therefore, given that \(\dim_\Lambda(M) = 0\), we get: \[d_A(\rk_\Lambda,\rk_{n,k}) = |\dim_\Lambda(M) - \dim_{n,k}(M)| = \frac{\log_p|M/(\omega_nM + p^kM)|}{kp^n}\,.\] Since \(M\) is torsion, then so is its localization \(M_{(p)}\). Over the discrete valuation ring \(\Lambda_{(p)}\), this means that \(p^{k_0}M_{(p)} = 0\) for some \(k_0 \geq 0\). Hence, for \(k \geq k_0\) we have 
\begin{align*}
    d_A(\rk_\Lambda, \rk_{(p),k}) &= |\dim_{\Lambda}(M) - \dim_{(p),k}(M)|\\
    &=\frac{\len_{\Lambda_{(p)}}(M_{(p)}/p^kM_{(p)})}{k}\\
    &= \frac{\len_{\Lambda_{(p)}}(M_{(p)})}{k} = \frac{\mu}{k}\,.
\end{align*}
The statement now follows from Theorem~\ref{thm-main-triangle-inequality}.
\end{proof}

It now only remains to estimate the quantity \(d_A(\rk_{(p),k}, \rk_{n,k})\) in the statement of Corollary~\ref{cor-Iwasawa-for-torsion}, which we will be able to do very explicitly if we assume that the module \(M\) presented by \(A\) has finite coinvariants. We start with a couple of lemmmas.

\begin{lem}\label{lem-maximal-finite-submodule} Let \(M\) be a finitely generated \(\Lambda\)-module, for each \(j \geq 0\) let \[M[(T,p)^j] = \left\{x \in M\colon (T,p)^jx = 0\right\}\] be the \(\Lambda\)-submodule of elements annihilated by the ideal \((T,p)^j\). Then, the submodule \[F(M) = \bigcup_{j \geq 1} M[(T,p)^j]\] is finite and contains every finite submodule of \(M\). Moreover, there exists \(n_0\) such that for any \(n \geq n_0\) we have \(\omega_n F(M) = 0\).
\end{lem}
\begin{proof} Since \(M\) is Noetherian, the chain \(M[(T,p)^j]\) stabilizes. Hence, there exists \(j\geq 1\) such that \(F(M) = M[(T,p)^j]\). As this is a finitely generated module over the finite ring \(\Lambda/(T,p)^j\), it is finite. Conversely, if \(N \leq M\) is a finite submodule, then its annihilator \(\operatorname{Ann}_\Lambda(N)\) is a finite index ideal of \(\Lambda\), and therefore contains \((T,p)^i\) for some \(i\geq 0\). The final assertion follows from the fact that \(\omega_n \to 0\) in the profinite topology of \(\Lambda\), which coincides with its \((T,p)\)-adic topology.
\end{proof}

\begin{lem}\label{lem-euler-characteristics} Let \(P\) be a finitely generated \(\Lambda\)-module annihilated by a power of \(p\), and denote by \(P[\omega_n]\) the submodule of \(P\) annihilated by \(\omega_n\). Then, for every \(n \geq 0\), both \(P/\omega_n P\) and \(P[\omega_n]\) are finite and \[\len_{\bbz_p}(P/\omega_nP) - \len_{\bbz_p}(P[\omega_n]) = p^n\len_{\Lambda_{(p)}}(P_{(p)})\,.\] Moreover, \(P[\omega_n] = F(P)\) for sufficiently large \(n\).
\end{lem}
\begin{proof} Let \(k\geq 1\) be such that \(p^k P = 0\). Then, both \(P/\omega_n P\) and \(P[\omega_n]\) are finitely generated modules over the finite ring \(\Lambda/(\omega_n,p^k) \simeq (\bbz/p^k\bbz)[C_{p^n}]\), and therefore are also finite. Denoting by \(\chi_n(P) = \len_{\bbz_p}(P/\omega_n P) - \len_{\bbz_p}(P[\omega_n])\), we see that \(\chi_n(P)\) is additive under short exact sequences: if \[0 \to P_1 \to P_2 \to P_3 \to 0\] is a short exact sequence of finitely generated \(\Lambda\)-modules annihilated by a power of \(p\), then applying the snake lemma to the commutative diagram \[\begin{tikzcd}
0 \arrow[r] & P_1 \arrow[d, "\omega_n\cdot"'] \arrow[r] & P_2 \arrow[d, "\omega_n\cdot"'] \arrow[r] & P_3 \arrow[d, "\omega_n\cdot"'] \arrow[r] & 0 \\
0 \arrow[r] & P_1 \arrow[r]                             & P_2 \arrow[r]                             & P_3 \arrow[r]                             & 0
\end{tikzcd}\] yields the exact sequence \[0 \to P_1[\omega_n] \to P_2[\omega_n] \to P_3[\omega_n] \to P_1/\omega_nP_1 \to P_2/\omega_nP_2 \to P_3/\omega_n P_3 \to 0\,.\] The alternating sum of their lengths is zero, and thus \(\chi_n(P_2) = \chi_n(P_1) + \chi_n(P_3)\).

For each \(i\), let \(Q_i = p^iP/p^{i+1}P\), so that \(\chi_n(P) = \sum_{i=0}^{k-1} \chi_n(Q_i)\). Observe that each \(\omega_n\) acts on \(Q_i\) as multiplication by \(T^{p^n}\). Each \(Q_i\) is a finitely generated \(\bbf_p[\![T]\!]\)-module, and if we write \(s_i = \dim_{\bbf_p(\!(T)\!)} Q_i \otimes_{\bbf_p[\![T]\!]} \bbf_p(\!(T)\!)\) we have an embedding \(\bbf_p[\![T]\!]^{s_i} \to Q_i\) with finite cokernel. Hence:
\begin{align*}
    \chi_n(Q_i) &= \chi_n(\bbf_p[\![T]\!]^{s_i}) + \chi_n(Q_i/\bbf_p[\![T]\!]^{s_i})\\
    &= p^ns_i\,,
\end{align*} where in the last equality we use that any endomorphism of a finite length module has kernel and cokernel with same lengths. Thus: \[\chi_n(P) = \sum_{i=0
}^{k-1} \chi_n(Q_i) = p^n \sum_{i=0}^{k-1}s_i = p^n \len_{\Lambda_{(p)}}(P_{(p)})\] by Lemma~\ref{lem-localization-at-p-structure}. The final claims follows from Lemma~\ref{lem-maximal-finite-submodule}.
\end{proof}

\begin{lem}[{cf. \cite[Lem. 2.4.6]{sharifi-iwasawa}, \cite[Lem. 2.48]{Ochiai2023Iwasawa1}}]\label{lem-scaling-lemma} Let \(L\) be a finitely generated \(\Lambda\)-module that is a free \(\bbz_p\)-module of finite rank, and assume that some power of \(T\) annihilates \(L/pL\). Then, there exists an integer \(n_0\) such that for all \(n \geq n_0\) we have \(\omega_{n+1}L = p\omega_nL\).
\end{lem}
\begin{proof} Let \(F\colon L \to L\) be the endomorphism induced by the action of \(1 +T\). Since the action of \(T\) is nilpotent modulo \(p\), we know that \(F^{p^{n_0}} \equiv \operatorname{id} \pmod{p^2}\) for some integer \(n_0 \geq 1\). For any \(n\geq n_0\), let \[S_n = \operatorname{id} + F^{p^n} + \cdots + F^{(p-1)p^n} \in \End_{\bbz_p}(L)\,.\] Then, \(S_n \equiv p\operatorname{id} \pmod{p^2}\), and so we can write \(S_n = pU_n\) with \(U_n \equiv \operatorname{id} \pmod{p}\). In particular, \(U_n\) is an automorphism of \(L\), and we have \[(F^{p^{n+1}}-\operatorname{id}) = (F^{p^n}-\operatorname{id})S_n = p(F^{p^n}-\operatorname{id})U_n\,.\]
Hence:
\begin{align*}
    \omega_{n+1}L &= (F^{p^{n+1}}-\operatorname{id})L\\
    &= p(F^{p^n}-\operatorname{id})U_nL\\
    &= p\omega_nL\,.\qedhere
\end{align*}
\end{proof}

We are now ready to finish estimating the bottom length of the Iwasawa triangle.

\begin{thm}\label{thm-bottom-edge} Let \(M \simeq \Lambda^r/A\Lambda^s\) be a finitely generated torsion \(\Lambda\)-module whose coinvariants \(M/\omega_nM\) are finite for all sufficiently large \(n\). Then, there exist integers \(\lambda \geq 0\), \(\nu\) and \(n_0\) such that for all \(n \geq n_0\) and all sufficiently large \(k\) (depending on \(n\)) we have \[d_A(\rk_{(p),k},\rk_{n,k}) = \frac{\lambda n + \nu}{kp^n}\,.\]
\end{thm}
\begin{proof} Observe that the hypothesis imply that \(M/\omega_nM\) is finite for \emph{every} \(n\geq 1\). For each \(i \geq 0\), let \(M[p^i]\) denote the \(\Lambda\)-submodule of \(M\) consisting of elements annihilated by \(p^i\), and set \[P = \bigcup_{i\geq 1} M[p^i]\,,\quad L = M/P\,.\] Then, \(L\) is torsion-free as a \(\bbz_p\)-module. Moreover, it is also finitely generated over \(\bbz_p\): it is also torsion and finitely generated as a \(\Lambda\)-module, so its annihilator \(\operatorname{Ann}_\Lambda(L)\) is non-zero. Taking any non-zero \(f \in \operatorname{Ann}_\Lambda(L)\) and writing \(f = p^auh\) by Theorem~\ref{thm-Weierstrass-prep} where \(a \geq 0\), \(u \in \Lambda^\times\) and \(h \in \bbz_p[T]\) is a distinguished polynomial, we see that \[fL = 0 \implies hL = 0\] by torsion-freeness over \(\bbz_p\). Hence, \(L\) is a quotient of \((\Lambda/(h))^r\), which is finitely generated over \(\bbz_p\) by Proposition~\ref{prop-Weierstrass-div}. Observe as well that some power of \(T\) acts trivially on \(L/pL\), as \(h \equiv T^{\deg h} \pmod{p}\).  Write \(L \simeq_{\bbz_p} \bbz_p^\lambda\) for \(\lambda \geq 0\), and observe that the inclusion \(P \to M\) induces an isomorphism \(P_{(p)} \simeq M_{(p)}\) of \(\Lambda_{(p)}\)-modules as \(h\) is a unit in \(\Lambda_{(p)}\).

Since \(L/\omega_n L\) is finite and \(L\) is a finitely generated free \(\bbz_p\)-module, we see that \(\ker(L \overset{\omega_n \cdot}{\to} L) = 0\). Applying the snake lemma to the commutative diagram \[\begin{tikzcd}
0 \arrow[r] & P \arrow[d, "\omega_n\cdot"'] \arrow[r] & M \arrow[d, "\omega_n\cdot"'] \arrow[r] & L \arrow[d, "\omega_n\cdot"'] \arrow[r] & 0 \\
0 \arrow[r] & P \arrow[r]                             & M \arrow[r]                             & L \arrow[r]                             & 0
\end{tikzcd}\] then yields the exact sequence \[0 \to P/\omega_nP \to M/\omega_nM \to L/\omega_nL \to 0\,.\] By the additivity of length under short exact sequences, we get \[\len_{\bbz_p}(M/\omega_nM) = \len_{\bbz_p}(P/\omega_nP) + \len_{\bbz_p}(L/\omega_nL)\,.\] Observe that \(P\) is annihilated by some power \(p^a\) of \(p\): since \(M\) is Noetherian, the ascending chain \(M[p^i]\) of submodules stabilizes. By Lemma~\ref{lem-euler-characteristics}, there exists an \(n_1\) such that we have \(F(P) = P[\omega_n]\) and \[\len_{\bbz_p}(P/\omega_nP) = p^n\len_{\Lambda_{(p)}}(P_{(p)}) + \len_{\bbz_p}(P[\omega_n])\] for any \(n \geq n_1\).

Fix for now an integer \(n \geq n_1\) and let \(k = k(n)\) be an integer large enough so that \(p^k(M/\omega_n M) = p^kP = 0\). We have:
\begin{align*}
    d_A(\rk_{(p),k},\rk_{n,k}) &= |\dim_{(p),k}(M) - \dim_{n,k}(M)|\\
    &= \left|\frac{\len_{\Lambda_{(p)}}(M_{(p)}/p^kM_{(p)})}{k} - \frac{\len_{\bbz_p}(M/(\omega_nM + p^kM))}{kp^n}\right|\\
    &= \left|\frac{\len_{\Lambda_{(p)}}(P_{(p)})}{k} - \frac{\len_{\bbz_p}(M/\omega_nM)}{kp^n}\right|\\
    &=  \left|\frac{\len_{\Lambda_{(p)}}(P_{(p)})}{k}- \frac{p^n\len_{\Lambda_{(p)}}(M_{(p)})}{kp^n} -  \frac{\len_{\bbz_p}(F(P))}{kp^n} - \frac{\len_{\bbz_p}(L/\omega_nL)}{kp^n}\right|\\
    &= \frac{\len_{\bbz_p}(F(P)) + \len_{\bbz_p}(L/\omega_nL)}{kp^n}
\end{align*}
As \(\len_{\bbz_p}(F(P))\) is a constant, it suffices to prove that \[\len_{\bbz_p}(L/\omega_nL) = \lambda n + C\] for some constant \(C\) and \(n\) greater than some constant \(n_0 \geq n_1\). By Lemma~\ref{lem-scaling-lemma}, we can choose \(n_0\) such that for any \(n \geq n_0 \geq n_1\) we have \(\omega_{n+1}L = p\omega_nL\). Hence:
\begin{align*}
    \len_{\bbz_p}(L/\omega_nL) &= \len_{\bbz_p}(\omega_{n_0}L/\omega_nL) + \len_{\bbz_p}(L/\omega_{n_0}L)\\
    &= \left(\sum_{m=n_0}^{n-1} \len_{\bbz_p}(\omega_mL/\omega_{m+1}L)\right) + \len_{\bbz_p}(L/\omega_{n_0}L)\\
    &= \left(\sum_{m=n_0}^{n-1} \len_{\bbz_p}(L/pL)\right) + \len_{\bbz_p}(L/\omega_{n_0}L)\\
    &= (n-n_0)\lambda + \len_{\bbz_p}(L/\omega_{n_0}L)\\
    &= \lambda n + C\text{ for }C = \len_{\bbz_p}(L/\omega_{n_0}L) -\lambda n_0\,.
\end{align*}
One can then take \(\nu = C + \len_{\bbz_p}(F(P))\).
\end{proof}

\begin{cor}\label{cor-Iwasawa-formula-arbitrary-module} Let \(M\) be a finitely generated torsion \(\Lambda\)-module whose coinvariants \(M/\omega_nM\) are finite for all sufficiently large \(n\). Then, there exist integers \(\mu \geq 0\), \(\lambda \geq 0\), \(\nu\) and \(n_0\) such that for all \(n \geq n_0\) we have \[\log_p |M/\omega_nM| = \mu p^n + \lambda n + \nu\,.\]
\end{cor}
\begin{proof} Let \(\lambda\), \(\nu\) and \(n_0\) be the integers from Theorem~\ref{thm-bottom-edge}, \(A\) be a matrix presenting \(M\) and take \(\mu = \len_{\Lambda_{(p)}}(M_{(p)})\). By Corollary~\ref{cor-Iwasawa-for-torsion}, for each \(n \geq n_0\) and sufficiently large \(k\) depending on \(n\) we have:
\begin{align*}
    \log_p|M/\omega_nM| &= \mu p^n + kp^n\cdot d_A(\rk_{(p),k},\rk_{n,k})\\
    &= \mu p^n + kp^n \cdot \frac{\lambda n + \nu}{kp^n}\\
    &= \mu p^n + \lambda n + \nu\,.\qedhere
\end{align*}
\end{proof}

\section{Totally ramified, single prime \texorpdfstring{\(\bbz_p\)}{Zp}-extensions}\label{sec-totally-ramified}

We can now turn to the proof of Theorem~\ref{thm-iwasawa-structure-theorem}. As such, we recall its notation: \(F\) is a number field, \(F = F_0 \subset F_1 \subset F_2 \subset \cdots\) is a tower of finite normal extensions of \(F\) such that for \(L = \bigcup_{i\geq 0} F_i\) we have \(G=\Gal(L/F) = \overline{\langle \gamma \rangle} \simeq \bbz_p\). For each \(n \geq 0\), \(L_n\) denotes the subfield fixed by \(G^{p^n} \simeq p^n\bbz_p\) and \(p^{e_n}\) is the highest power of \(p\) dividing the order of the class group \(\Cl(L_n)\). We let \(T = \gamma - 1 \in \Zalgebra{G}\), \(\Lambda = \Zalgebra{G} = \Zalgebra{T}\), \(\omega_n = \gamma^{p^n} - 1\) and for any abelian group \(M\) we denote by \(M(p)\) its \(p\)-primary part.

Bearing in mind the statement of Corollary~\ref{cor-Iwasawa-formula-arbitrary-module}, the ideal scenario would be to find a torsion \(\Zalgebra{G}\)-module \(M\) for which we have \(|M/\omega_nM| = |\Cl(L_n)(p)|\) whenever \(n\) is large enough. Unfortunately, there is no natural candidate for this \(M\) in general (see (4) of Theorem~\ref{thm-GCFT}). There is, however, a situation in which this module does exists: whenever there exists a single prime \(v\) on \(L\) over \(p\), and this prime is totally ramified. 

It seems to be a tradition to first prove his formula in the totally ramified case, and then to follow through with general \(\bbz_p\)-extensions (see \cite[Sec. 3.1.2]{Ochiai2023Iwasawa1}). This pedagogical choice also dates back to Iwasawa's paper \cite{iwasawa1959gamma}, which discusses a simplification of the proof under this particular hypothesis at the end. It is not the case that general extensions reduce to the totally ramified case, much to the contrary: the totally ramified case illustrates very clearly the mechanism behind Iwasawa's formula, and the presence of other primes and residue-field extensions only introduce technicalities one has to work around. As such, we keep on with this tradition. Hence, consider the following hypothesis: \\ \

\paragraph{\bfseries(HYP)} \(F\) has a unique prime above \(p\), which is totally ramified in \(L/F\).

\begin{proof}[Proof of Theorem~\ref{thm-iwasawa-structure-theorem} assuming {\upshape\bfseries(HYP)}] Let \(v\) be the unique prime of \(L\) over \(p\). In particular, for each \(n \geq 0\), its restriction \(u_n\) to \(L_n\) is the unique prime of \(L_n\) over \(p\) as well. The assumption that \(L/F\) is totally ramified at \(p\) means that for each \(n\) the inertia groups \(I_{v,n} = I(v/u_n)\) of \(v\) over \(L_n\) equals the whole Galois group \(G^{p^n}\), and the map \(\Sigma_n\colon I_{v,n} \to G^{p^n }\) of Theorem~\ref{thm-GCFT} is the identity. Hence, if \(X = \varprojlim \Cl(L_n)(p)\) denotes the inverse limit of the \(p\)-primary part of the class groups along the norm maps, by part (4) of Theorem~\ref{thm-GCFT} we have an isomorphism \(X/\omega_nX \simeq \Cl(L_n)(p)\) for each \(n \geq 0\). 

By Nakayama's lemma (\cite[Prop. 1.2.5]{Souza2026Thesis}), a profinite \(\Lambda\)-module \(M\) is finitely generated if and only if \(M/(T,p)M\) is finite. Since \(T = \omega_0\), we see that \(X/(\omega_0,p)X \simeq \Cl(F)(p)/p\Cl(F)(p)\) is finite by part (1) of Theorem~\ref{thm-GCFT} and thus \(X\) is finitely generated. By Proposition~\ref{prop-Harris-theorem}, we have \[\dim_\Lambda X = \lim_{n \to \infty} \frac{\dim_{\bbq_p} (X/\omega_nX) \otimes_{\bbz_p} \bbq_p}{p^n} = 0\,,\] so \(X\) is torsion. By Corollary~\ref{cor-Iwasawa-formula-arbitrary-module}, there are integers \(\mu \geq 0\), \(\lambda \geq 0\) and \(\nu\) such that for large enough \(n\) we have \[e_n = \log_p|X/\omega_nX| = \mu p^n + \lambda n + \nu\,.\qedhere\]
\end{proof}

As established in the proof of Corollary~\ref{cor-Iwasawa-for-torsion} and Theorem~\ref{thm-bottom-edge}, in this case we can express the constants \(\mu\) and \(\lambda\) in terms of the following ``rank-theoretical'' invariants of \(X\): \[\mu = \len_{\Lambda_{(p)}}(X_{(p)})\,, \quad \lambda = \dim_{\bbq_p} (X \otimes_{\bbz_p} \bbq_p)\,.\]

\section{General \texorpdfstring{\(\bbz_p\)}{Zp}-extensions}\label{sec-general-extensions}

Keeping with the notation of Section~\ref{sec-totally-ramified}, we now drop the hypothesis on how \(p\) splits in \(F\) or ramifies in \(L\). By parts (3) and (5) of Theorem~\ref{thm-GCFT}, there exists an integer \(m\geq 0\) such that all ramified primes in \(L/L_m\) are totally ramified and the maps \(X \to \Cl(L_n)(p)\) are surjective. Let \(\ffp_1\,,\ldots\,,\ffp_s\) be the primes in \(L_m\) which ramify in \(L/L_m\). Thus, for any \(n \geq m\), any prime \(v\) of \(L_n\) that ramifies in \(L/L_n\) lies above a unique \(\ffp_i\), and there is only one prime of \(L_n\) above each \(\ffp_i\). We can therefore label the ramified primes of \(L/L_n\) by \(v_{n,1}\,,\ldots\,,v_{n,s}\), and in particular their quantity is bounded independently of \(n\).

For each \(v_{n,r}\), its inertia subgroup \(I_{v_{n,r}}\) equals the whole Galois group \(G^{p^n}\), and thus \(\Sigma_n\colon \bigoplus_{r=1}^s I_{v_{n,r}} \to G^{p^n}\) is simply the addition map \(\bbz_p^s \to \bbz_p\). In particular, \(\ker \Sigma_n \simeq \bbz_p^{s-1}\) and \(\coker \Sigma_n = 0\). By part (4) of Theorem~\ref{thm-GCFT}, we get the exact sequence \begin{equation}\label{eq-final-descent} \bbz_p^{s-1} \to X/\omega_nX \to \Cl(L_n)(p) \to 0\,.
\end{equation}

Let \(Y_n = \ker(X \to \Cl(L_n)(p))\) and \(Y = Y_m\). Recall from part (5) of Theorem~\ref{thm-GCFT} that for \(\eta_{n,m} = \dfrac{\omega_n}{\omega_m}\) we have \(Y_n = \eta_{n,m}Y\). In particular, for any \(n \geq m\), we get two more short exact sequences \begin{equation}\label{eq-ses-Y-X-Cl}0 \to \eta_{n,m}Y \to X \to \Cl(L_n)(p) \to 0\,\end{equation}\begin{equation}\label{eq-ses-Y-CL-CL}
0 \to Y/\eta_{n,m}Y \to \Cl(L_n)(p) \to \Cl(L_m)(p) \to 0\,.\end{equation} The short exact sequence~(\ref{eq-ses-Y-CL-CL}) implies that \(Y/\eta_{n,m}Y\) is finite, while the exact sequence~(\ref{eq-final-descent}) implies that \(Y/\omega_nX\) is finitely generated as a \(\bbz_p\)-module.

\begin{lem} The \(\Lambda\)-module \(X\) is finitely generated and torsion.
\end{lem}
\begin{proof} For each \(n \geq m\), let \(\Lambda_n = \Zalgebra{G^{p^n}}\). It suffices to prove that \(X\) is finitely generated as a \(\Lambda_m\)-module. Since it is profinite, by Nakayama's lemma (\cite[Prop. 1.2.5]{Souza2026Thesis}) we once again see that it is finitely generated if and only if \(X/(\omega_m,p)X\) is finite. Tensoring the exact sequence~(\ref{eq-final-descent}) with \(\bbf_p\) over \(\bbz_p\) yields the exact sequence \[\bbf_p^{s-1} \to X/(\omega_m,p)X \to \Cl(L_m)/p\Cl(L_m) \to 0\,.\] Hence, \(X/(\omega_m,p)X\) is finite, and so \(X\) is finitely generated.

Now, tensoring the exact sequence~(\ref{eq-final-descent}) with \(\bbq_p\) over \(\bbz_p\) yields the exact sequence \[\bbq_p^{s-1} \to (X/\omega_nX) \otimes_{\bbz_p} \bbq_p \to 0\,,\] and so Proposition~\ref{prop-Harris-theorem} gives us \[\dim_\Lambda(X) = \lim_{n \to \infty} \frac{\dim_{\bbq_p}((X/\omega_nX) \otimes_{\bbz_p} \bbq_p)}{p^n} \leq \lim_{n \to \infty} \frac{s-1}{p^n} = 0\,.\] This proves that \(X\) is torsion.
\end{proof}

In particular, all the \(\Lambda\)-submodules \(Y_n\) of \(X\) are also finitely generated and torsion. We are now ready to prove the main theorem of this note. We will accomplish this by repeating the proof of Lemma~\ref{lem-euler-characteristics}, Lemma~\ref{lem-scaling-lemma} and Theorem~\ref{thm-bottom-edge} with \(\eta_{n,m}\) playing the role of \(\omega_n\), with the due changes noted.

\begin{proof}[Proof of Theorem~\ref{thm-iwasawa-structure-theorem}] Write \[P = \bigcup_{i\geq 0} Y[p^i]\,,\quad N = Y/P\,.\] Since \(Y\) is Noetherian, the chain \(Y[p^i]\) stabilizes, and thus there exists some \(a \geq 0\) such that \(p^aP = 0\). As in the proof of Theorem~\ref{thm-bottom-edge}, \(N\) is finitely generated and free as a \(\bbz_p\)-module, say of rank \(\lambda\).

Since \(Y/\eta_{n,m}Y\) is finite by the short exact sequence~(\ref{eq-ses-Y-CL-CL}), we see that \(\ker(N \overset{\eta_{n,m}\cdot }{\to} N) = 0\). Applying the snake lemma to the commutative diagram \[\begin{tikzcd}
0 \arrow[r] & P \arrow[r] \arrow[d, "{\eta_{n,m}\cdot}"'] & Y \arrow[r] \arrow[d, "{\eta_{n,m}\cdot}"'] & N \arrow[r] \arrow[d, "{\eta_{n,m}\cdot}"'] & 0 \\
0 \arrow[r] & P \arrow[r]                                 & Y \arrow[r]                                 & N \arrow[r]                                 & 0
\end{tikzcd}\] yields the short exact sequence \[0 \to P/\eta_{n,m}P \to Y/\eta_{n,m}Y \to N/\eta_{n,m}N \to 0\,.\] Hence: \begin{equation}\label{eq-final-1}e_n - e_m=\log_p|Y/\eta_{n,m}Y| = \log_p|P/\eta_{n,m}P| + \log_p|N/\eta_{n,m}N|\,.\end{equation}
The proof of Lemma~\ref{lem-euler-characteristics} applies verbatim with \(\eta_{n,m}\) in place of \(\omega_n\), and since \(\eta_{n,m} \equiv T^{p^n-p^m} \pmod{p}\) it yields a constant \(n_1 \geq m\) for which whenever \(n \geq n_1\) we have \(P[\eta_{n,m}] = F(P)\) and \begin{equation}\label{eq-final-2}\log_p|P/\eta_{n,m}P| = \len_{\Lambda_{(p)}}(P_{(p)}) (p^n-p^m) + \log_p|F(P)|\,.\end{equation}

We now finish as in Theorem~\ref{thm-bottom-edge} and Lemma~\ref{lem-scaling-lemma}. For \(n \geq n_0 \geq n_1\) where \(n_0\) is some constant, the element \[\eta_{n+1,n} = 1 + \gamma^{p^n} + \cdots + \gamma^{(p-1)p^n}\] acts on \(N\) as an endomorphism \(S_{n+1}\) satisfying \(S_{n+1}\equiv p\operatorname{id} \pmod{p^2}\), and so we can write \(S_{n+1} = pU_{n+1}\) with \(U_{n+1} \equiv \operatorname{id} \pmod p\). In particular, \(U_{n+1}\) is an automorphism of \(N\) and we have that \(\eta_{n+1,m} = \eta_{n+1,n}\eta_{n,m}\) satisfies \(\eta_{n+1,m}N = p\eta_{n,m}N\). These give
\begin{equation}\label{eq-final-3}\begin{split}
    \log_p|N/\eta_{n,m}N| &= \log_p|\eta_{n_0,m}N/\eta_{n,m}N| + \log_p|N/\eta_{n_0,m}N|
    \\ &=\left(\sum_{k=n_0}^{n-1} \log_p|\eta_{k,m}N/\eta_{k+1,m}N|\right) + \log_p|N/\eta_{n_0,m}N|\\
    &= \left(\sum_{k=n_0}^{n-1} \log_p|N/pN|\right) + \log_p|N/\eta_{n_0,m}N|\\
    &= (n-n_0)\lambda + \log_p|N/\eta_{n_0,m}N|\,.
\end{split}
\end{equation}
We take \[\mu = \len_{\Lambda_{(p)}}(P_{(p)})\,,\quad \lambda = \dim_{\bbq_p}(N\otimes_{\bbz_p} \bbq_p)\]\[\text{and }\nu = \log_p|F(P)| + \log_p|N/\eta_{n_0,m}N| - n_0\lambda + e_m -\mu p^m\,.\] By combining the equations~(\ref{eq-final-1}),~(\ref{eq-final-2}) and~(\ref{eq-final-3}) we obtain \[e_n= \mu p^n+ \lambda n + \nu\] whenever \(n \geq n_0\).
\end{proof}

\begin{rmk}[{cf. \cite[Rmk. 2.51]{Ochiai2023Iwasawa1}}] One could give uniform proof of Theorem~\ref{thm-iwasawa-structure-theorem} analogous to the totally ramified, single prime case if one first proves the results of Section~\ref{sec:iwasawa-triangle} in a more general setting, at the expense of working with less explicit rank functions. Concretely, fix an integer \(m \geq 0\) and a distinguished polynomial \(g\) dividing \(\omega_m\), and for \(n\geq m\) put \(f_n = \dfrac{\omega_n}{g}\). This includes both the families \(f_n = \omega_n\) with \(g = 1\) and \(f_n = \eta_{n,m}\) with \(g = \omega_m\). One then replaces the rank functions \(\rk_{n,k}\) by those induced from the finite length module \(\Lambda/(f_n,p^k)\). The following results are then shown by very similar arguments:
\begin{itemize}
    \item (Lemma~\ref{lem-topology-iwasawa-triangle}) For \(k\) fixed, \(\lim_{n \to \infty} \rk_{n,k} = \rk_{(p),k}\).
    \item (Theorem~\ref{thm-main-triangle-inequality}) \(\rk_{n,k}(A) \leq \rk_{(p),k}(A) \leq \rk_\Lambda(A)\) for any matrix \(A\) over \(\Lambda\).
    \item (Lemma~\ref{lem-scaling-lemma}) Let \(L\) be a finitely generated \(\Lambda\)-module that is a free \(\bbz_p\)-module of finite rank, and assume that some power of \(T\) annihilates \(L/pL\). Then, there exists an integer \(n_0\) such that for all \(n \geq n_0\) we have \(f_{n+1}L = pf_nL\).
\end{itemize}
This allows one to establish a version of Corollary~\ref{cor-Iwasawa-formula-arbitrary-module} for the growth of the modules \(M/f_nM\) of the form \[\log_p |M/f_nM| = \mu p^n + \lambda n + \nu\] for large enough \(n\). Applying this to \(g = 1\) and \(M = X\) recovers Theorem~\ref{thm-iwasawa-structure-theorem} under the hypothesis {\bfseries(HYP)} of Section~\ref{sec-totally-ramified}, while applying it to \(g = \omega_m\) and \(M = Y\) recovers the argument in this section.
\end{rmk}

\bibliographystyle{alpha}
\bibliography{biblio}

@misc{sharifi-iwasawa,
  author       = {Sharifi, Romyar},
  title        = {{Iwasawa} Theory},
  howpublished = {Lecture notes, \url{https://www.math.ucla.edu/~sharifi/notes/iwasawa.pdf}},
  note         = {Accessed: 2026-10-19}
}

@article{iwasawa1959gamma,
  author  = {Iwasawa, Kenkichi},
  title   = {On {$\Gamma$}-extensions of algebraic number fields},
  journal = {Bulletin of the American Mathematical Society},
  year    = {1959},
  volume  = {65},
  number  = {4},
  pages   = {183--226},
  doi     = {10.1090/S0002-9904-1959-10317-7}
}

@phdthesis{Souza2026Thesis,
  author  = {Souza, Henrique},
  title   = {Ranks over {pro-$p$} groups},
  school  = {Universidad Aut{\'o}noma de Madrid},
  year    = {2026},
  address = {Madrid, Spain},
  url     = {https://iq.mat.br/resources/MendesdaSilvaeSouza_HenriqueAugusto_tesis.pdf}
}

@book{Ochiai2023Iwasawa1,
  author    = {Ochiai, Tadashi},
  title     = {{Iwasawa} Theory and Its Perspective, Volume 1},
  series    = {Mathematical Surveys and Monographs},
  volume    = {272},
  publisher = {American Mathematical Society},
  address   = {Providence, RI},
  year      = {2023},
  isbn      = {978-1-4704-5672-6},
  doi       = {10.1090/surv/272}
}

@book{Neukirch1999,
  title = {Algebraic Number Theory},
  ISBN = {9783662039830},
  ISSN = {0072-7830},
  url = {http://dx.doi.org/10.1007/978-3-662-03983-0},
  DOI = {10.1007/978-3-662-03983-0},
  journal = {Grundlehren der mathematischen Wissenschaften},
  publisher = {Springer Berlin Heidelberg},
  author = {Neukirch,  J\"{u}rgen},
  year = {1999}
}

@book{Washington1997,
  title = {Introduction to Cyclotomic Fields},
  ISBN = {9781461219347},
  ISSN = {0072-5285},
  url = {http://dx.doi.org/10.1007/978-1-4612-1934-7},
  DOI = {10.1007/978-1-4612-1934-7},
  journal = {Graduate Texts in Mathematics},
  publisher = {Springer New York},
  author = {Washington,  Lawrence C.},
  year = {1997}
}

@article{Harris2000,
  title = {Correction to p-adic Representations Arising from Descent on Abelian Varieties},
  volume = {121},
  ISSN = {1570-5846},
  url = {http://dx.doi.org/10.1023/A:1001730616194},
  DOI = {10.1023/a:1001730616194},
  number = {1},
  journal = {Compositio Mathematica},
  publisher = {Cambridge University Press (CUP)},
  author = {Harris,  Michael},
  year = {2000},
  month = Mar,
  pages = {105–108}
}

\end{document}